\documentclass[sts]{imsart}
\usepackage{graphicx}
 \DeclareGraphicsExtensions{.eps, .ps}
\usepackage{soul,color}
\usepackage{amsmath}
\usepackage{amsthm}
\usepackage{amsfonts}
\usepackage{newfloat}

\font\tenbf=cmb10
\numberwithin{equation}{section}
\theoremstyle{plain}
\newtheorem{theorem}{Theorem}[section]

\newtheorem{remark}{Remark}
\newtheorem{corollary}[theorem]{Corollary}

\newtheorem{definition}[theorem]{Definition}

\newtheorem{lemma}[theorem]{Lemma}
\newtheorem{proposition}[theorem]{Proposition}

\def\bft{{\bf t}}
\newcommand\ci{\perp\!\!\!\perp}
\def\pr{\mathop{\rm pr}\nolimits}

\def\subdot{{\hbox{\tenbf .}}}

\def\indep{\mathrel{\rlap{$\perp$}\kern1.6pt\mathord{\perp}}}
\def\Real{\mathbb{R}}
\def\binomial#1#2{\left(\begin{array}{c} {#1} \\ {#2} \end{array} \right)}
\def\ascf#1#2{{#1}^{\uparrow #2}}

\def\given{\,|\,}

\def\P{{\cal P}}

\def\T{{\bf T}}
\def\Ttx{\mathop{\rm Ttx}\nolimits}
\def\Tx{\mathop{\rm Tx}\nolimits}

\def\OP{{\cal OP}}
\def\OF{{\cal OF}}
\def\F{{\cal F}}

\def\beginignoretext{\setbox0=\vbox\bgroup}
\def\endignoretext{\egroup}
\newcount\qno\qno=0
\def\question{\medbreak\noindent \advance\qno by 1 {\bf Q.~\number\qno} \ }

\usepackage{tikz}
\usepackage{tikz-qtree}

\begin{document}

\begin{frontmatter}
\title{The pilgrim process}
\runtitle{The pilgrim process}

\begin{aug}
\author{\fnms{Walter} \snm{Dempsey}\ead[label=e1]{wdem@uchicago.edu}}
\and
\author{\fnms{Peter} \snm{McCullagh}\ead[label=e2]{pmcc@galton.uchicago.edu}}
\runauthor{W.Dempsey \& P. McCullagh}

\affiliation{University of Michigan and University of Chicago}

\address{Department of Statistics, University of Michigan, 
		1085 S. University Ave., Ann Arbor, MI 48103 \printead{e1}}
\address{Department of Statistics, University of Chicago, 
		5734 S. University Ave., Chicago, IL 60637\printead{e2}}


\end{aug}

\begin{abstract}
Pilgrim's monopoly is a probabilistic process giving rise to
a non-negative sequence $T_1, T_2,\ldots$ that is infinitely exchangeable,
a natural model for time-to-event data.
The one-dimensional marginal distributions are exponential.
The rules are simple, the process is easy to generate sequentially,
and a simple expression is available for both the joint density
and the multivariate survivor function.
There is a close connection with
the Kaplan-Meier estimator of the survival distribution.
Embedded within the process is an infinitely exchangeable 
ordered partition processes connected to Markov branching processes
in neutral evolutionary theory.
Some aspects of the process, such as the distribution of the number of blocks,
can be investigated analytically and confirmed by simulation.  
By ignoring the order, the embedded process can be considered
as an infinitely exchangeable partition process, shown to 
be closely related to the Chinese restaurant process.   
Further connection to the Indian buffet process is also provided.  
Thus we establish a previously unknown link between
the well-known Kaplan-Meier estimator and the important Ewens sampling formula.
\end{abstract}


\end{frontmatter}

\section{Introduction}

A fundamental principle underlying statistical modeling is that arbitrary choices
such as sample size or observation labels should not affect the sense of a model
and meaning of parameters.  For example, a priori the labels assigned to 
observations (unit~$1$, $2$, \ldots ) carry no meaning other than to 
distinguish among units.  This guiding precept gives rise to the study of 
exchangeable distributions.
Infinite exchangeability captures the idea that the data generating process
should not be affected by arbitrary relabeling of the sample.  
Here we are concerned with infinitely exchangeable non-negative sequences
$(T_1, T_2, \ldots)$, which arise in 
survival analysis where the observations are event times.
Pilgrim's monopoly is a sequential description of
such a process.  
Restriction to the subsample $[n] = \{ 1, \ldots, n \}$ yields $T[n] = (T_1, \ldots, T_n)$.
For~$T_i$ and $T_j$, the probability of equality~$\pr (T_i = T_j)$ can be 
non-negative, implying the number of unique times in the set~$T[n]$ 
may be less than~$n$.
In this paper, we outline both existing and new theoretical results
concerning the associated asymptotic behavior of the process.  
For example, the number of unique times grows
asymptotically at a rate proportional to $\log^2 (n)$.  
An argument is given as to why the pilgrim's 
monopoly is the most fitting model for applied work 
when such ties may be the result of numerical rounding. 

Research on nonparametric Bayesian survival analysis has been based around 
the de Finnetti approach to constructing exchangeable survival processes
by generating survival times conditionally independent and identically
distributed given a completely independent hazard measure,
i.e.~the cumulative conditional hazard is a L\'evy process
(Cornfield and Detre, 1977; Kalbfleisch, 1978; Hjort, 1990; Clayton, 1991).
Such processes are sometimes called \emph{neutral to the right}
(Doksum, 1974; James, 2006).
Pilgrim's monopoly corresponds to a particular choice of L\'evy process
and thus fits naturally within the existing literature.
However, the probability density can be computed using the characteristic index,
and the process can be simulated directly from the predictive distributions, 
allowing us to bypass the L\'evy process entirely.

While connections to survival analysis were expected, 
we find that the pilgrim process provides a fundamental link 
among seemingly disparate areas of statistics.  
The connections arise due to the embedding of two
processes within the pilgrim's monopoly.
First the ties of $T[n]$ generate
a partial ranking (or ordered partition) of the units.
That is, $T[n]$ generates a sequence of unique times~
$t_1 < t_2 <\ldots< t_k$ for some $k \leq n$.  This induces
a clustering of units by unique times~(i.e.
$B_i = \{ j  \in [n] \given T_j = t_i \}$).
The partial ranking of the blocks is obtained by the ordering
of the associated unique times.    The embedded ordered
partition process is infinitely exchangeable as it only depends
on block order and block size.  

The embedded ordered partitions are characterized by a
splitting rule~$q_n (d) = q(n-d, d)$, which only depends on the total number 
individuals at risk leading up to the next failure time,~$n$, and the number of those who fail at that time,~$d \in \{1,\ldots, n \}$.
An interesting connection to Markov branching
models in neutral evolutionary theory is shown, in particular the family
of beta-splitting models (Aldous, 1996). These contain several
important models, namely the symmetric binary tree model, 
Yule model, and uniform model. 
We explain how the beta-splitting model corresponding to the 
pilgrim process may be most fitting in the study of phylogenetic
trees.

Infinitely exchangeable partition processes are important in cluster analysis,
serving as a natural prior for the partition of sampled units.
More recently there has been interest in a generalization to feature allocation
models (Broderick, Pitman, and Jordan, 2013).  
Each infinitely exchangeable process on non-negative sequences
induces an infinitely exchangeable partition process by the relation
$i \sim j$ if $T_i = T_j$.
We show how this induced partition process contains the well-known 
Chinese restaurant process.
Moreover, a generalization of the pilgrim process to doubly 
indexed sequences $\{ T_{i,j} \} $ is shown to be connected with
the Indian buffet process.

\subsection{Outline}

The goal of this paper is to highlight existing and new results related to the
pilgrim process.  We establish a fundamental connection between this process
and models in seemingly unrelated areas of statistics.  
In particular, we provide a previously unknown connection 
between the Kaplan-Meier estimator 
and the Ewens sampling formula.

We organize our discussion as follows.  In Section~\ref{pilgrim}, we provide 
a description of the pilgrim's monopoly.  In Section~\ref{illustration}, we 
investigate the behavior of the process via simulation, providing some 
intuition to the key features.  In Section~\ref{key_behavior}, we 
summarize the known theoretical behavior of the process 
as well as establish several new results.  In particular
we argue that the pilgrim process is the natural choice for survival analysis
as it is the unique subfamily of processes with weakly continuous predictive
distributions.  In Section~\ref{literature}, we review the connections with 
previous literature, in particular the study of \emph{neutral to the right processes}
used in Bayesian survival analysis and Markov branching models in the study of 
probability distributions on cladograms. In Section~\ref{extensions}, we discuss
extensions of the pilgrim process to a richer set of time-to-event models with various
asymptotic behavior, as well as to the setting of recurrent events.  
In Section~\ref{clustering}, we provide an explicit connection between the 
various processes and the Chinese restaurant process and Indian buffet
process.  The connection yields alternate methods for producing 
exchangeable random partition and feature allocation processes.

\section{Pilgrim's monopoly} \label{pilgrim}
Although not as simple as the Chinese restaurant process (Pitman, pp.~57--62,~2006)
or the Indian buffet process (Griffiths and Ghahramani,~2005),
the rules of pilgrim's monopoly are straightforward and
reminiscent of the board game even though the available real estate is an infinite
straight line rather than a square loop.
Pilgrim's monopoly is a novel description of a process first described by
Dempsey and McCullagh (2015) in the more general study of Markov survival processes.
It is a process involving a sequence of pilgrims or travellers
who pay toll fees and hotel taxes as they go, proceeding until their funds are depleted.
The initial funds $X_1, X_2,\ldots$ for each pilgrim are distributed independently 
according to the exponential distribution with unit mean.
Toll fees are paid continuously at the posted per-mile rate $\tau(s)$,
which is reduced after each passing traveller,
and taxes are levied by each hotel encountered on the route.
If his funds are exhausted, the traveller establishes a new hotel at the point of exhaustion,
and collects taxes from subsequent passers-by.
If he arrives at a hotel where his remaining funds are insufficient
to pay the tax, the funds are forfeit,
and he remains at the hotel as a resident.  

Initially, there are no hotels, the toll rate is uniform $1/\rho$ dollars per mile,
so the first traveller setting out from the origin
proceeds to the point $T_1 = \rho X_1$ where he establishes the first hotel.
At that stage, the toll rate at point $s < T_1$ is reduced to $1/(1+\rho)$,
the rate at $s > T_1$ is unchanged,
and the hotel tax is the log ratio $\log((1+\rho)/\rho)$ of the
upstream toll rate to the downstream rate.
Upstream is the direction of travel.

After $n$ pilgrims have set out from the origin
and reached their destinations, $T_1,\ldots, T_n$,
hotels have been established at unique points $0 < t_1 < t_2 < \cdots < t_k$ with $k \leq n$.
Hotel~$r$ contains $d_r \ge1$ pilgrims, so that $\sum d_r = n$.
Let $R(s)$ be the number of travellers who have proceeded beyond point~$s > 0$,
i.e.,~$R(s) = \#\{i \le n \colon T_i > s\}$, so that $n - R(s)$ is the
number of travellers whose destination lies in $(0, s]$.
The toll rate $\tau(s) = 1/(\rho + R(s))$ is right-continuous and piecewise constant,
increasing from $1/(\rho+R(t_{r-1}))$ immediately before
hotel~$r$ to $1/(\rho + R(t_r))$ immediately after.
As always, the hotel tax is the log ratio
$\log(\tau(s) / \tau(s^-))$ of upstream to downstream toll rates.
Thus, the sequence of tolls and taxes faced by pilgrim $n+1$
on the first three legs of his journey are as follows:
\[
\arraycolsep=10pt
\begin{array}{lll}
\hbox{Interval} & \hbox{Total toll} & \hbox{Hotel tax} \cr
\noalign{\smallskip}
(0, t_1) &\displaystyle\frac{ t_1}{n+\rho} &\displaystyle{\log\biggl(\frac{n+\rho}{n+\rho - d_1} \biggr)} \cr
\noalign{\smallskip}
(t_1, t_2) &\displaystyle\frac{t_2 - t_1}{n+\rho - d_1} &\displaystyle{\log\biggl(\frac{n+\rho-d_1}{n+\rho - d_1-d_2}
\biggr)}\cr 
\noalign{\smallskip}
(t_2, t_3) &\displaystyle\frac{t_3 - t_2}{R(t_2) + \rho} &\displaystyle\log\biggl(\frac{R(t_2) + \rho}{R(t_3) + \rho} \biggr) \cr 
\end{array}
\]
where $R(t_2) = n - d_1 - d_2$.
The destination $T_{n+1} > 0$ is the point of maximum progress
permitted by the funds available,
i.e.,~the total tolls and taxes are such that
$\Ttx((0, T_{n+1})) \le X_{n+1} \le \Ttx((0, T_{n+1}])$.
If $X_{n+1} < t_1/(n+\rho)$, the pilgrim establishes a new hotel at $T_{n+1} = (n+\rho) X_{n+1}$;
otherwise he pays the full toll for the segment $(0, t_1)$.
If his remaining funds are sufficient to cover the hotel tax at $t_1$,
he does so, and his
pilgrimage continues in the same way until his funds are exhausted.
Otherwise, his destination is $T_2=t_1$.

\begin{figure}[ht]
\begin{center}
\includegraphics[height=10cm,width=12.5cm, angle=0]{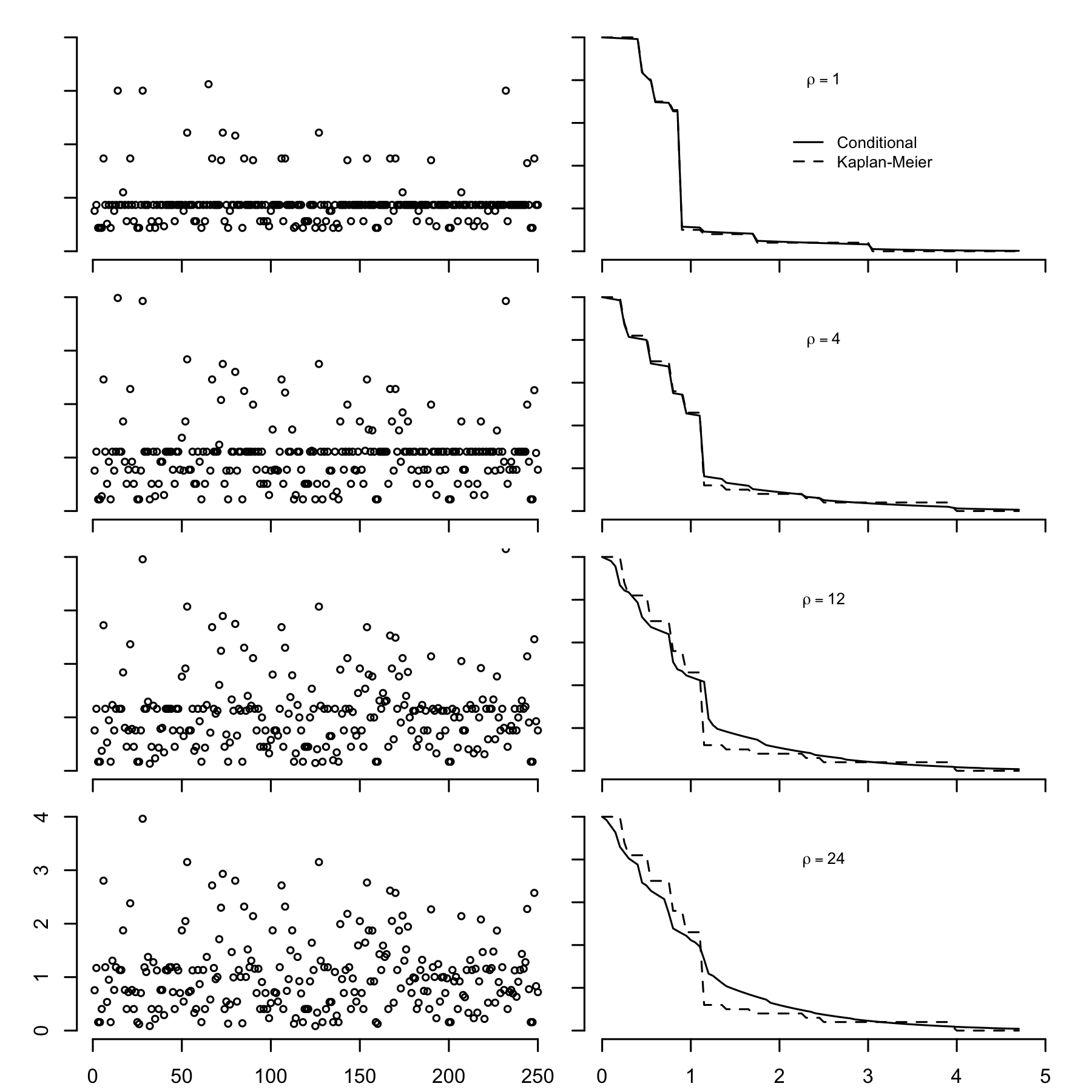}
\end{center}
\vspace{-0.5cm}
\caption{Simulated pilgrim process for $\rho=1, 4, 12, 24$.
The left panels show the sequence $\rho T_i$ for 250 pilgrims;
the right panels show the conditional survival distribution
$\pr(T_{51} > t \given T[50])$ (solid line) together with the Kaplan-Meier curve (dotted line),
both based on the first 50 values.
}
\label{simulated_pilgrim}
\end{figure}

\section{Illustration by simulation} \label{illustration}
The pilgrim process takes an input sequence $X_1, X_2,\ldots$,
and produces an output sequence $T_1, T_2,\ldots$ such that
$T[n]$ is a deterministic function of~$X[n]$, though not an invertible one.
The following is an example of a 10-component pair $X, T$ generated with $\rho=1$:
\begin{eqnarray*}
X &=& (0.36, 0.25, 0.36, 2.24, 0.40, 0.03, 1.17, 1.68, 3.31, 1.24, 0.35, 0.50,\ldots) \\
T_\rho (X) &=& (0.36, 0.36, 0.36, 1.12, 0.36, 0.18, 0.36, 0.85, 1.89, 0.85, 0.36, 0.36,\ldots),
\end{eqnarray*}
where $X$ is exact and $T$ is given to two decimal places.
In other words, the transformation $X\mapsto T$ is non-stochastic,
the only random element being the stochastic nature of the inputs.

The process is straightforward to simulate, but a simple algorithm
with minimal book-keeping is likely to be considerably less efficient than
an algorithm that keeps track of hotel positions and occupancy numbers,
updating them as needed.
For greater generality, one may include a second parameter $\nu > 0$ such that
the toll rate is $\tau(s) = \nu/(\rho + R(s))$, which does not affect the taxes.
Although not entirely obvious, the transformation $X\mapsto T_{\nu,\rho}(X)$ is such that
$\nu T_{\nu,\rho}(X) = T_{1,\rho}(X)$, so the effect is merely multiplicative. 
A reasonable argument may be made for taking $\nu \propto \rho$ as the default.
In particular, the toll after the final hotel would be $\lambda := \nu / \rho$ and 
limits as $\rho$ tends to $0$ or $\infty$ become well-defined.
However, unless otherwise stated, we take $\nu=1$.

The left panels of Figure~\ref{simulated_pilgrim} show four realizations of the process with the same input sequence, $n=250$, and parameters $\rho=1, 4, 12, 24$ respectively.
The right panels show the conditional survival distribution (solid line) for which
\[
-\log \pr(T_{m+1} > t \given T[m]) = \Ttx((0, t])
\]
is the sum of the posted tolls and taxes to be levied on one pilgrim,~$m=51$,
travelling from the origin to~$t+\epsilon$.

Considering only the taxes, we find that
\begin{eqnarray*}
S_n(t, \rho) = e^{-\Tx((0, t])} &=& \prod_{i : t_i \le t} \frac{\rho+R(t_i)}{\rho+R(t_i) + d_i}\\
S_n(t, 0) &=& \prod_{i : t_i \le t} \frac{R(t_i)}{R(t_i) + d_i}
\end{eqnarray*}
so that the limit as $\rho\to 0$ is the Kaplan-Meier survival curve.
In other words, $-\log S_n(t, 0)$ is the limit as $\rho\to 0$
of the sum of the hotel taxes posted for pilgrim~$n+1$ in $(0, t]$.
This is also shown in Fig.~1 for comparison, using only the first 50 values.
The sample size $m=50$ was kept sufficiently small to make the differences visible.
The Kaplan-Meier distribution has mass only at the hotels,
and the limiting tax at the final hotel is infinite.
(In the presence of censoring, the total mass or tax may be finite,
i.e.,~the Kaplan-Meier distribution may have positive mass at infinity,
and the Kaplan-Meier curve then differs from the limiting conditional survival distribution
only in the final interval following the last hotel.)

Exchangeability is apparent in the sense that any fixed permutation
of the sequence would look much the same.
It is also apparent that the process for small~$\rho$ is much more grainy
than the process for large~$\rho$.
The number of distinct values, i.e.,~the number of hotels,
in the four simulations is 15, 32, 74, and 100 respectively.

\begin{figure}[ht]
\begin{center}
\includegraphics[height=8cm,width=12.5cm, angle=0]{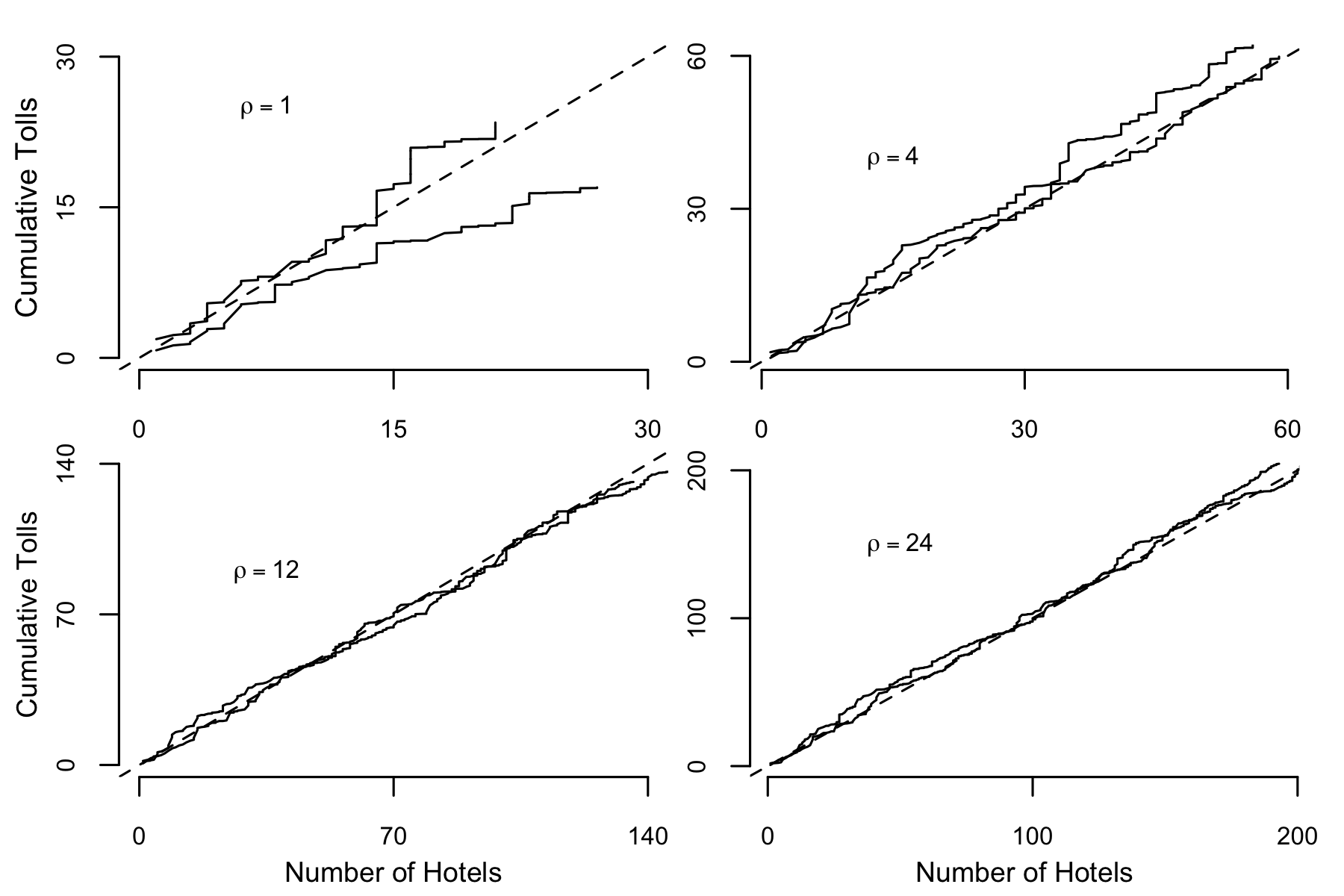}
\end{center}
\vspace{-0.5cm}
\caption{Cumulative tolls plotted against the number of hotels established.
Two replicates are shown per panel. The parameter values are $\rho=1, 4, 12, 24$ for the four pairs.
}
\label{cumtoll_vs_numhotel}
\end{figure}

Simulation allows us to keep track not only of hotel development
and hotel occupancy, but also the collection of tolls and taxes.
The $r$th pilgrim travelling through the sector $(t, t+dt)$
contributes $dt/(r-1+\rho)$ in tolls, so the total toll collected
in this sector is $Z(dt) = \zeta(R(t))\, dt$,
where $\zeta(n) = \sum_{j=0}^{n-1} 1/(\rho+j)$ and $\zeta(0) = 0$.
The total collected in tolls from all pilgrims passing through
the inter-hotel zone $(t_{i-1}, t_i)$ is $\zeta(R(t_{i-1})) (t_i - t_{i-1})$,
and the sum of these tolls is
$Z = \int_0^\infty \zeta(R(t))\, dt$.
The remainder $X_\subdot - Z$ is distributed in taxes and forfeits to the various hotels.
In other words, the process generates not only a partition of $[n]$ into
$k$~blocks which are ordered in space, but also a partition of
$X_\subdot$ by tolls and taxes into $2k$~parts that are also
linearly ordered along the route.

It is of some interest to examine the relationship between the number
of hotels,~$k$, and the total paid in tolls.
Both are increasing in~$n$.
Figure~\ref{cumtoll_vs_numhotel} suggests that they are essentially proportional for large~$n$,
with limiting ratio one independent of~$\rho$.   Specifically, the limiting ratio is equal to the parameter $\nu$.

\begin{figure}[ht]
\begin{center}
\includegraphics[height=8cm,width=12.5cm, angle=0]{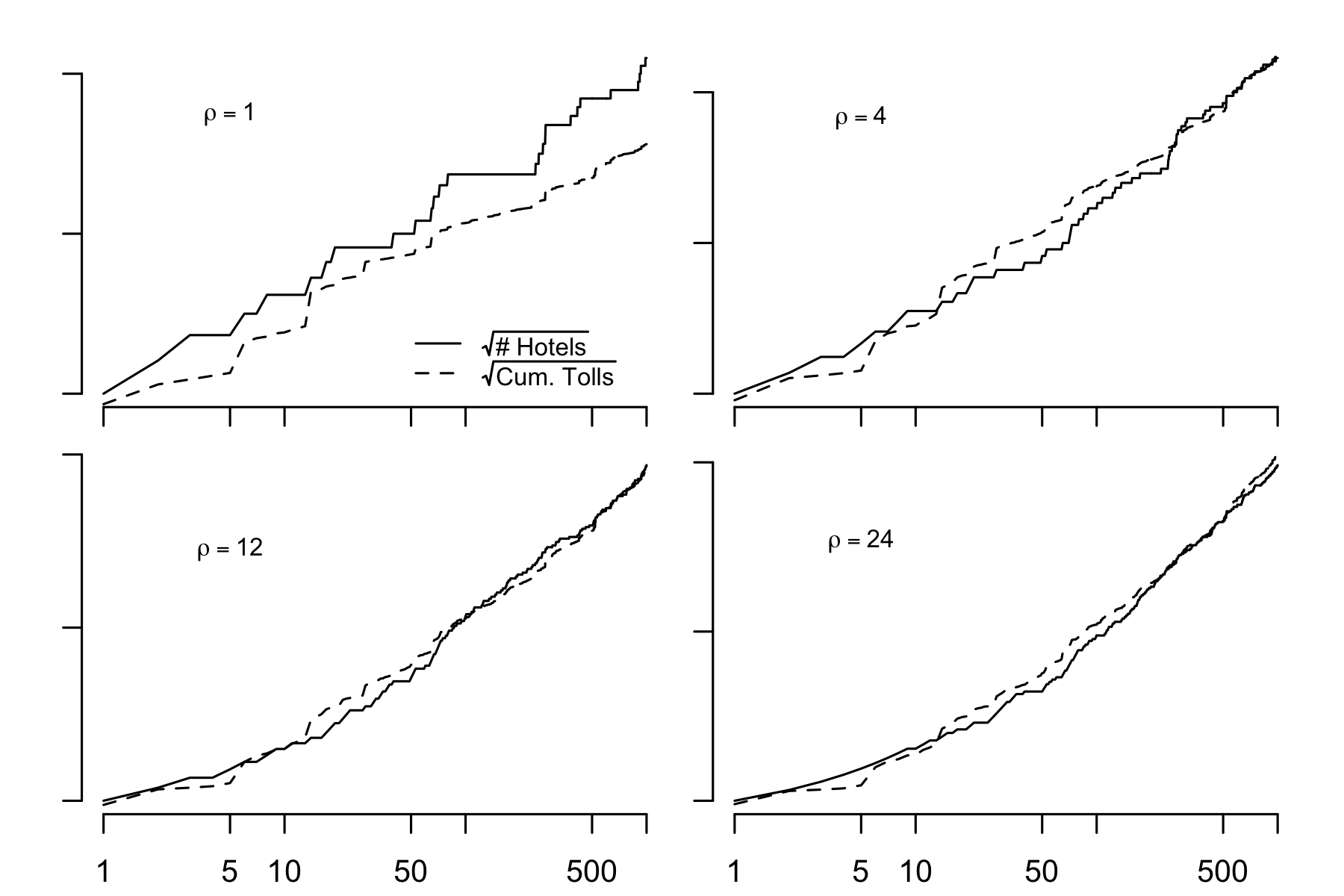}
\end{center}
\vspace{-0.5cm}
\caption{Square root of the number of hotels and square root of
accumulated tolls versus $\log(n)$.
The parameter values are $\rho=1, 4, 12, 24$ for the four panels.
}
\label{hot_tolls_vs_logn}
\end{figure}

Figure~\ref{hot_tolls_vs_logn} is a plot of $k^{1/2}$ against $\log(n)$, i.e.,~the square root
of the number of hotels against the log of the number of pilgrims,
one panel for each parameer value $\rho=1, 4, 12, 24$.
Also plotted is the square root of the accumulated tolls against $\log(n)$.
The simulation is in agreement with the claim that $k \propto \log^2(n)$,
at least for sufficiently large $n$.

Simulations also allow us to monitor the accumulated wealth of hotel proprietors,
either in temporal order of establishment or in spatial order relative to the origin.  The wealth of a hotel is the sum of the forfeits of its occupants plus the taxes paid by passers-by. Figure~\ref{tolls_dist_to_ownr} shows the total tax collected by each proprietor in temporal order of establishment (left) and spatial order (right).  In each case, there is an upper boundary linearly decreasing in the index.  The distribution of total tax collected is approximately uniform under temporal ordering from zero to the upper boundary.  Figure~\ref{tolls_dist_amg_prop} shows the total tax when wealth is distributed among occupants.  In this case,  the maximal wealth per occupant appears to grow initially and then decrease, reaching its peak at the median index independent of $\rho$.

\begin{figure}[ht]
\begin{center}
\includegraphics[height=8cm,width=12.5cm, angle=0]{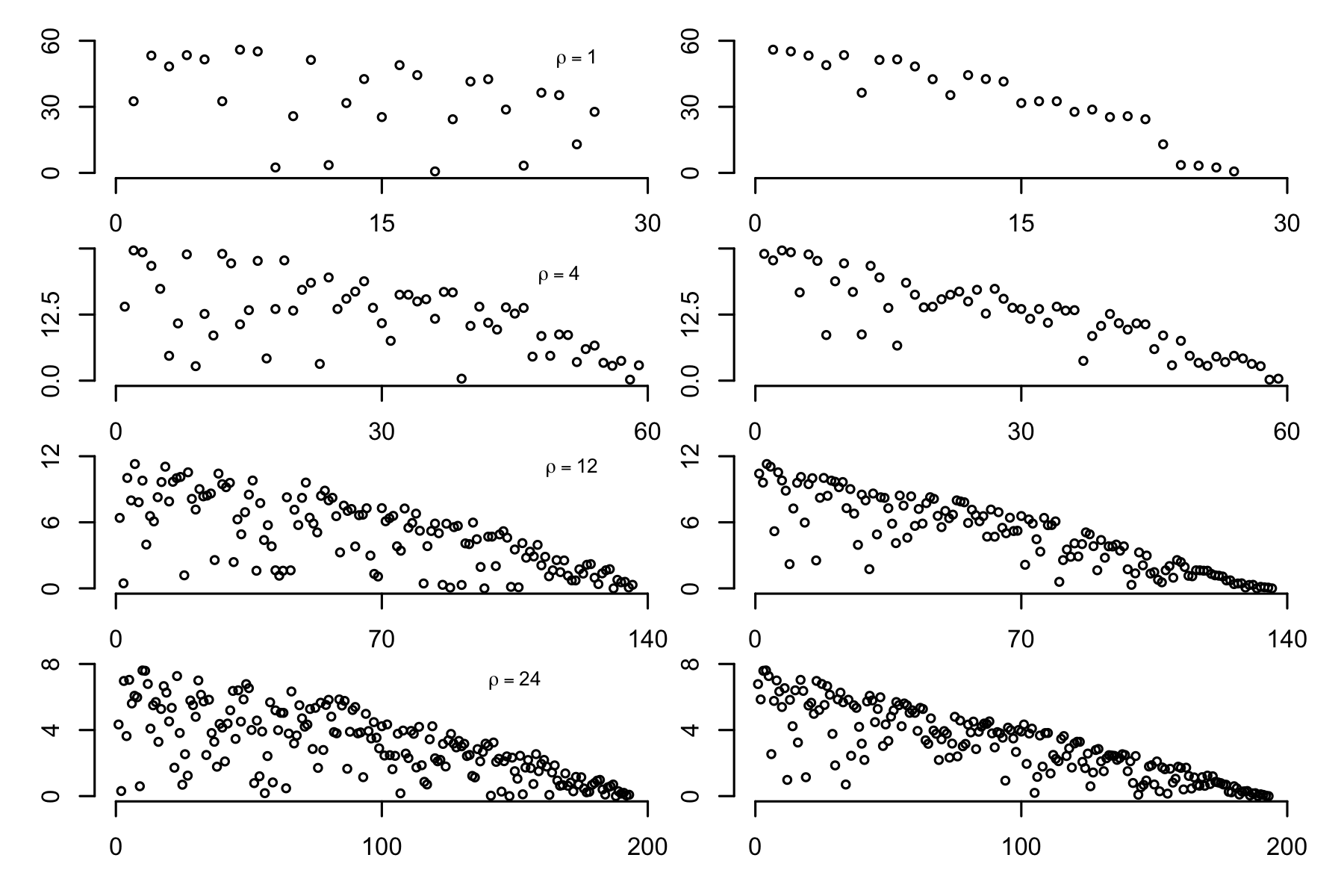}
\end{center}
\vspace{-0.5cm}
\caption{Total taxes per hotel proprietor ordered temporally (left) and spatially (right).
The parameter values are $\rho=1, 4, 12, 24$ for the four panels with $n = 1000$.
}
\label{tolls_dist_to_ownr}
\end{figure}

Finally, we investigate the size of hotels in both spatial and temporal order.  Figure~\ref{size_of_hotels} is a plot of the total number of residents in each hotel on a log-scale for various choices of $\rho$ under both orderings.  As $\rho$ tends to infinity, the number of residents in each hotel tends to one.  At the other extreme, as $\rho$ tends to zero, all individuals occupy the same hotel.  For intermediate values, there tends to be few large hotels with the remainder of moderate size.  The number of hotels of size $1$ is approximately Poisson with rate proportional to $\log (n)$. Under the spatial order relative to the origin, the larger hotels tend to exist closer to the origin.  Under temporal ordering, the larger hotels occur earlier; however, the bunching around early indices is less pronounced under this ordering.  The number of hotels is independent of the parameter $\nu$.

\section{Key behavior}
\label{key_behavior}
The pilgrim process is a mathematical exercise that can be studied at various levels
and in various ways.
The questions of interest are mostly related to the behaviour of the
sequence $T_1, \ldots, T_n$ both for finite~$n$ and the limit as $n\to \infty$.
Some very natural questions are related to the number and the sizes of the blocks,
i.e.,~the hotel occupancy numbers $(d_1,\ldots, d_k)$,
possibly, but not necessarily, ignoring order.
Other questions are concerned with expenditures, particularly the expenditure on 
taxes and forfeits versus the expenditure on tolls.
The total expenditure on tolls, $Z=\int_0^\infty \zeta(R(s))\, ds$,
is a function of $T_1,\ldots, T_n$, but the remainder is not.
As a function of~$n$, it appears that $Z$
is closely related to the number of blocks.

\newpage

\begin{figure}
\begin{center}
\includegraphics[height=9cm,width=12.5cm, angle=0]{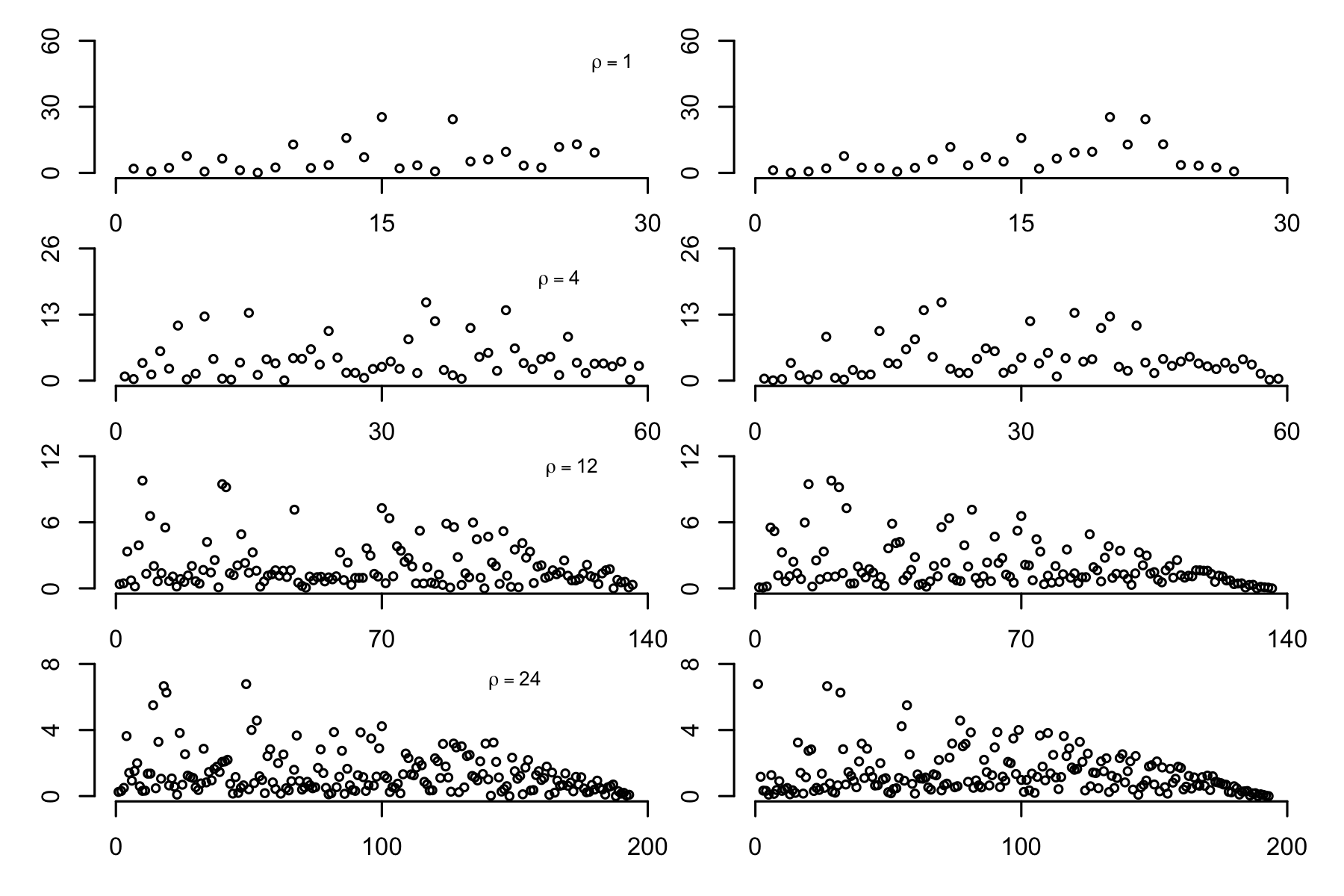}
\end{center}
\vspace{-0.5cm}
\caption{Wealth distributed among occupants versus temporal (left) and spatial (right) ordering.
The parameter values are $\rho=1, 4, 12, 24$ for the four panels for $n = 1000$.
}
\label{tolls_dist_amg_prop}
\end{figure}

\begin{figure}
\begin{center}
\includegraphics[height=9cm,width=12.5cm, angle=0]{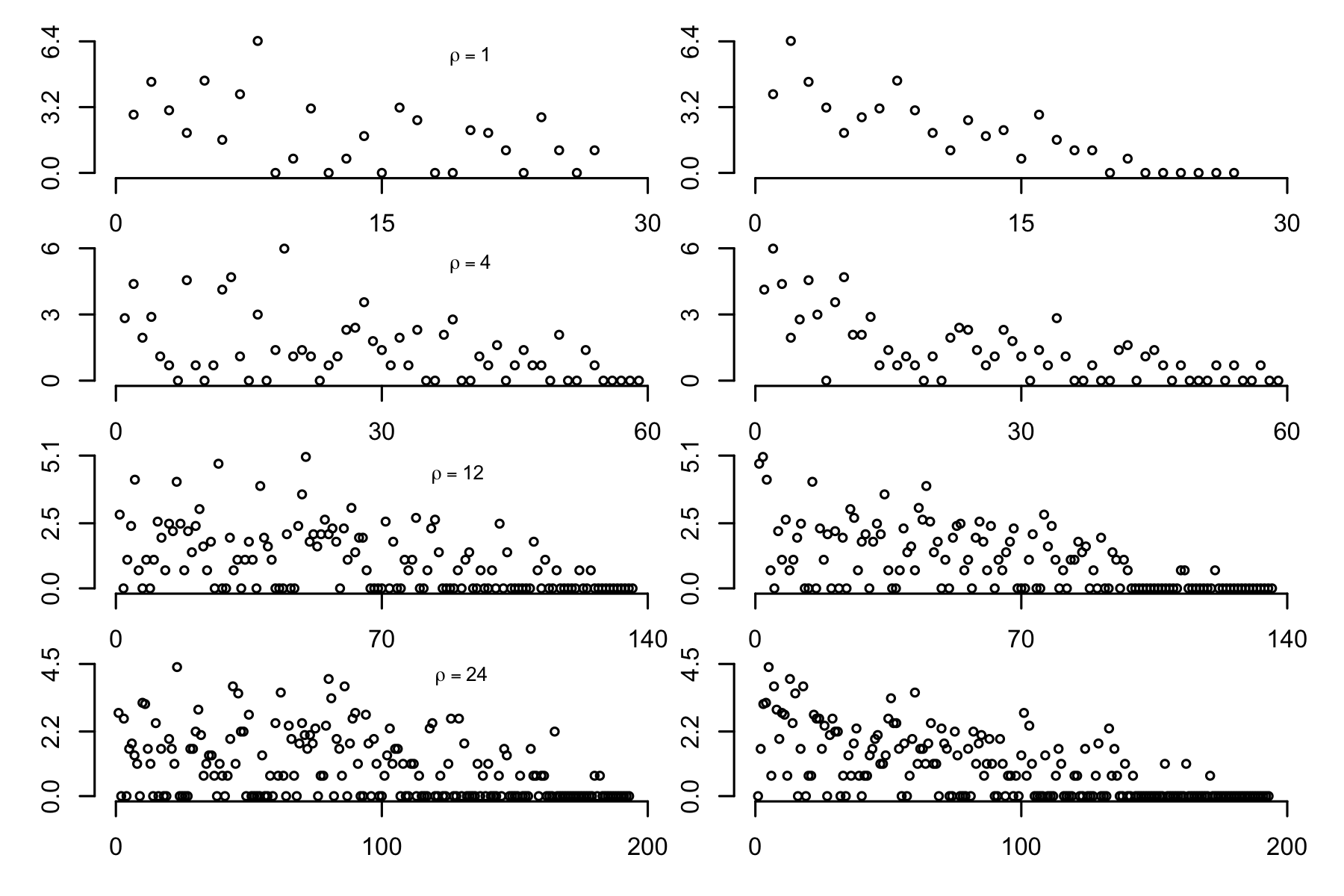}
\end{center}
\vspace{-0.5cm}
\caption{Log number of occupants per hotel versus temporal (left) and spatial (right) ordering.
The parameter values are $\rho=1, 4, 12, 24$ for the four panels for $n = 1000$.
}
\label{size_of_hotels}
\end{figure} 

\newpage

Here we review some of the key existing results found in Dempsey and McCullagh (2015).  

\begin{lemma}
The joint density function for the non-negative
sequence ${\T_n} = (T_1, \ldots, T_n)$ is
\begin{equation} \label{density}
\ascf \rho n f_n({\T_n}) d\T_n = \exp\Bigl(-\int_0^\infty \zeta(R_{\T_n}(s))\, ds\Bigr)
\prod_{r=1}^k \Gamma(d_r)
\end{equation}
where $R_{\T_n} (s) = \#\{i\le n \colon T_i > s\}$ and
$d_r = \# \{ i \le n : T_i = t_r )$ where $t_{r}$ is the~$r$th unique
time of $\T_n$ ($0 < t_1 < \ldots < t_k$ with $k \leq n$).
\end{lemma}

\begin{remark} \normalfont
The conditional distribution of $T_{n+1}$ given $T[n]= \T_n = (T_1,\ldots, T_n)$
has discrete atoms at each of the hotels $t_1,\ldots, t_k$, plus a remainder
that is continuous on $(0,\infty)$.
In particular, the conditional distribution is such that
\[
-\log\pr(T_{n+1} > t \given T[n]) = \hbox{tolls+taxes payable in }(0, t].
\]
Moreover, the conditional distribution given $T[n]$ is such that
$T_{n+1} \sim \min(X, X')$
where $X, X'$ are conditionally independent,
$X > 0$ is distributed continuously as a function of the tolls, 
and $X'$ is distributed on $\{t_1,\ldots, t_k, \infty\}$ as a function of taxes.
\end{remark}

\subsection{Weak continuity of predictions}

The conditional distribution of $T_{n+1}$ given
the sequence $\T_n=(T_1,\ldots, T_n)$ of previous failures
has an atom at each distinct failure time, and is continuous elsewhere.
The predictive distribution is weakly continuous if a small
perturbation of the failure times gives rise to a small
perturbation of the predictive distribution.
In other words, for each $n\ge 1$, and for each non-negative vector~$\bft$,
\[
\lim_{\epsilon\to 0} \pr(T_{n+1}\le x \given T[n]=\bft+\epsilon) =
\pr(T_{n+1} \le x \given T[n]=\bft)
\]
at each continuity point, i.e.,~$x> 0$ not equal to one of the failure times.

\begin{remark}
\label{rmk_cty}
Consider two adjacent hotels that are $\epsilon$-apart.  Suppose they are at located at 
$t_i$ and $t_{i+1} = t_i + \epsilon$ for some $i < k$ with $d_{i}$ and $d_{i+1}$ 
occupants respectively.  Then as $\epsilon \to 0$, we say the hotels ``merge'' and
there is one hotel at $t_i^\prime = t_i$ with $d_i^\prime = d_i + d_{i+1}$ occupants.
Weak continuity implies as hotels merge, the total taxes 
collected is the same as the sum of the taxes collected at each hotel individually.
\end{remark}

Remark~\ref{rmk_cty} shows that the pilgrim process has 
weakly continuous predictive distributions as hotel taxes are the log ratio of upstream
and downstream toll rates.  Theorem~\ref{cty_pred_thm}
shows that among a large class of exchangeable survival processes, 
this property is unique to the pilgrim process.

\begin{theorem}[Dempsey and McCullagh, 2015] 
\label{cty_pred_thm}
A Markov survival process has weakly continuous predictive distributions
if and only if it is a pilgrim process:
$\zeta(n) \propto  \psi(n+\rho) - \psi(\rho)$ for some $\rho > 0$.
The iid exponential model is included as a limit point.
\end{theorem}

A detailed derivation of this can be found in Appendix~\ref{app:cty_pred}.
Theorem~\ref{cty_pred_thm} provides justification for the use of the 
pilgrim process in applied work where observed tied failure times
are often the result of numerical rounding.

\subsection{Number of blocks \& block sizes}

Lemma~\ref{num_blocks} proves the relationship between the 
mean number of hotels and the number of pilgrims.

\begin{lemma}[Dempsey and McCullagh, 2015]
\label{num_blocks}
The number of blocks, i.e.,~the number of distinct components in $T_1,\ldots, T_n$,
is a random non-decreasing function of~$n$ 
whose mean satisfies the recursion
\begin{equation} \label{recurrence}
\mu_n = 1 + \frac1{\zeta(n)} \sum_{d=1}^n \frac{\mu_{n-d} } {d+\rho-1}.
\end{equation}
In particular, $\mu_n \propto \log^2(n)$ for large $n$. 
\end{lemma}

\begin{remark}
The total expenditure on tolls by the first $n$ pilgrims is 
$Z = \int_0^\infty \zeta(R(s))\, ds$,
so that $X_{\subdot} - Z$ is the tax levied by hotels.
The mean value, 
$E_n [ Z ] $, also satistifes equation~\eqref{recurrence}.  
Therefore, $E_n [Z] \to \mu_n$ as $n$ tends to infinity.
\end{remark}

A question of interest is how the
number of hotels of a given size grows as a function
of the number of pilgrims.
Below is a new result highlighting this fact for
the pilgrim process.

\begin{lemma} 
\label{pilg_pois_lemma}
The number of hotels of size $j$ is approximately Poisson 
with asymptotic rate proportional to $\log (n)/j$.
\end{lemma}

The proof can be found in Appendix~\ref{app:pilg_poisson}.

\section{Connections with the literature}
\label{literature}

\subsection{Bayesian survival analysis}
The output of the pilgrim process is an infinitely exchangeable sequence,
which automatically has a de~Finetti representation as a conditionally
independent and identically distributed sequence.
It is in fact one of a large class of similar survival processes
investigated in detail by Hjort (1990), in which
$\Lambda$ is a stationary, completely independent, measure on $\Real$, and
$T_1,T_2\ldots$ are conditionally independent given $\Lambda$,
with distribution
\[
\pr(T_i > t \given \Lambda) = e^{-\Lambda((0, t])}.
\]
This means that $\Lambda$ is a L\'evy process with infinitely divisible increments,
which are necessarily stationary if the marginal distributions are to be exponential.
Processes of this type have been investigated by
Doksum (1974),
Ferguson and Phadia (1979),
Kalbfleisch (1978),
Clayton (1991)
and
James (2006).
They are sometimes called {\it neutral to the right.}
In general, the increments need not be stationary, there may be individual-specific weights $w_i$
such that $\pr(T_i > t \given \Lambda) = e^{- w_i \Lambda((0, t])}$,
and the observations may be subject to censoring.
It is important that these matters be accommodated in applied work,
and the processes in this class have the virtue that inhomogeneities and right-censoring are easily accommodated.
We have chosen not to emphasize this aspect of things because it tends
to obscure the underlying process.

Each process in the class is associated with an infinitely
divisible distribution $X = \Lambda((0,1))$ which has a characteristic exponent
$\zeta(t) = -\log(E(e^{-t X}))$ for $t \ge 0$, and an associated L\'evy measure.
This means that the multivariate survival function is
\[
\pr(T_1 > t_1,\ldots, T_n > t_n)
	= E\bigl(e^{-\Lambda(0, t_1] - \cdots - \Lambda(0, t_n]}\bigr)
	= \exp\Bigl(-\int_0^\infty\! \zeta(R_t(s))\, ds\Bigr).
\]
For example, the characteristic exponent and the L\'evy measure
for the gamma process considered by
Kalbfleisch (1978) and Clayton (1991) are
\[
\zeta^\star(t)=\log(1 + t/\rho)\quad\hbox{and}\quad w^\star(dz) = z^{-1} e^{-\rho z}\, dz.
\]
The characteristic exponent and the L\'evy measure for the pilgrim process
are 
\[
\zeta(t) = \psi(\rho+t) - \psi(\rho) \quad\hbox{and}\quad
w(dz) = e^{-\rho z}\, dz / (1 - e^{-z}),
\]
where $\psi$ is the derivative of the log gamma function.
These two pairs $\zeta,\zeta^\star$ and $w, w^\star$ are very alike
and they produce very similar processes,
but the limits as $\rho\to 0$ are not exactly the same.
The obvious attraction of the pilgrim process is that it has joint and conditional distributions that are simple to evaluate,
which does not, in itself, make it more suited for applied work than any of the other processes in the class.
Other processes in the same class include $\zeta(t) = t^\alpha$ for $0 < \alpha \le 1$,
which is quite different because it is not analytic at the origin.

In general, the joint density of $T[n]$ at times $T\in\Real^n$ may be written in terms of
the characteristic exponent
\begin{equation}\label{density}
\exp\Bigl(-\int_0^\infty \zeta(R_T(s))\, ds \Bigr) \times
\prod_{r=1}^k (-1)^{d_r-1} (\Delta^{d_r} \zeta)(R_t(t_r)),
\end{equation}
where $0 < t_1 <\cdots < t_k$ are the distinct components of~$t$,
$d_1,\ldots, d_k$ are their multiplicities, and
$\Delta^d \zeta$ is the $d$th order forward difference
\[
(\Delta^d \zeta)(x) = \sum_{j=0}^d (-1)^{d-j} \binomial d j \zeta(x+j). 
\]
The conditional density of $T_{n+1}$ given $T[n]=t$ is most easily described in terms
of the conditional hazard measure, which has an atom or tax
\begin{eqnarray*}
-\log \frac{(\Delta^{d_r} \zeta)(R_t(t_r)+1)} {(\Delta^{d_r} \zeta)(R_t(t_r))}
\end{eqnarray*}
at each distinct component of~$t$,
plus a continuous component with density
$(\Delta \zeta)(R_t(s))$ at $s > 0$.  A detailed description is provided
by Dempsey and McCullagh (2015).
In general, the hotel tax is not equal to the log ratio of upstream to
downstream toll rates.

Expression (\ref{density}) is also correct under right censoring provided that the product
is restricted to failures.
It is also correct under the inhomogeneous model with weights~$w_i$, provided that
$\zeta(R(t))$ is replaced with $\zeta(w(R(t)))$, and the forward differences at $t_r$ are
re-defined in the obvious way as an alternating sum over subsets of the $d_r$ residents at~$t_r$.
It is also correct under monotone temporal transformation provided that the
integral with respect to Lebesgue measure is transformed to an integral with respect
to another measure~$\nu$, and the Jacobian is included.
These modifications are sufficient to cover the great majority of the non-exchangeable
or non-homogeneous processes in this class.

\subsection{Neutral population genetics}

The non-negative sequence generated by the pilgrim process
can be thought of as a simple binary-tree embedded in continuous
time.  At each failure time~$t_i$ the remaining units are split
into those who fail,~$D_i$, and the set of units that remain alive,~$R_i$.
The tree-like structure is reminiscent of evolutionary tree
models studied in neutral evolutionary theory.
In particular, the pilgrim process has much in common with the beta-splitting
model (Aldous, 1996) introduced as a natural family for the
study of probability distributions on cladograms.  

A \emph{cladogram} is a visual representation of the relations 
among various species.  Unlike the trees associated with
the pilgrim process, the cladogram has no distinction
between right and left edges and are not embedded
in continuous time.  Figure~\ref{fig:cladogram} illustrates such
a tree; the figure suggests $A$ and $B$ are more closely related
compared to $\{ C, D,E \}$ and among these $D$ and $E$ are 
more closely related.

The beta-splitting model assumes the probability of splitting $n\geq 2$ individuals 
into two groups of size~$i \in \{ 1, \ldots, n-1 \}$ and~$n-i$ respectively is given by
\[
q_n (i) = \frac{1}{a_n (\beta) } \frac{\Gamma(\beta+ i ) \Gamma(\beta + n - i)}{\Gamma (i+1) \Gamma(n-i+1)}
\]
where $a_n (\beta)$ is the normalizing constant and $\beta > -1$. 
The beta-splitting model contains several fundamental examples, notably
the Yule model ($\beta = 1$), the uniform model ($\beta = -1/2$), and
the symmetric binary model as a limit ($\beta \to \infty$).
Examining equation~\eqref{splitting_rules} shows the obvious similarity
to our setting.  The key difference is that the splitting rule for 
non-negative sequences need not be symmetric. Moreover 
the number of failures at each failure time needs to be positive, 
but the risk set can be empty, $R_i = \emptyset$.

\begin{figure}
\centering
\begin{tikzpicture}[every tree node/.style={draw,circle},
   level distance=1.25cm,sibling distance=.5cm, 
   edge from parent path={(\tikzparentnode) -- (\tikzchildnode)}]
\Tree [ \edge node[auto=right] {$t_0$};
    [
    \edge node[auto=right] {$t_1$}; 
    [  
      [
      	[.A  ] [.B ]
      ]
    ]
    \edge node[auto=left] {$t_1^\prime$}; 
    [
    	[
    		[.C ]
    	]
       \edge node[auto=left] {$t_2$};
    	[
      		[.D ] [.E ]
      	]
    ] ]
    ]
\end{tikzpicture}
\caption{Example cladogram with associated holding times}
\label{fig:cladogram}
\end{figure}
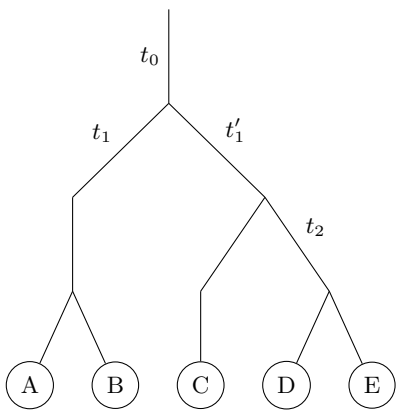

A question of interest is the conditional distribution 
of a new particle given an existing cladogram.  Given
an existing split of~$n$ particles into $i$ left and $n-i$ right particles,
the probability of the new particle going right
is
\begin{equation} \label{cladcond1}
\frac{n-i+\beta}{n + 2 \beta}.
\end{equation}
for $\beta \in (-2, \infty)$.
Now suppose starting with~$n$ particles there are $i$ consecutive left-singleton
splits, $\{ (1,n-1), (1,n-2), \ldots \}$.  Then the probability that the
particle goes right at each of these splits is
\begin{equation} \label{cladcond2}
\prod_{j=1}^i \frac{n-j+\beta}{n-j+1 + 2 \beta}.
\end{equation}
for $\beta \in (-2, \infty)$.
Equations~\eqref{cladcond1} and~\eqref{cladcond2}
are equal for $\beta = 0$, the latter smaller for $\beta >0$ 
and greater for $\beta < 0$.  

We can embed such a cladogram in continuous time by associating
with each edge a random holding time (see McCullagh, Pitman, and Winkel, 2008).
Figure~\ref{fig:cladogram} illustrates the holding times until each subsequent split.
Given $n$ particles, the holding time until the next fragmentation
is distributed exponential with rate parameter $\lambda_n$.  The 
sequence~$(\lambda_n)_{n\geq 2}$ must satisfy 
\[
\lambda_{n+1} ( 1 - q_n (1) ) = \lambda_n
\]
in order for the process to remain consistent and exchangeable.

To particle~$i$ is associated the non-negative sequence of
holding times~$(T^{(i)}_{1}, \ldots, T^{(i)}_{k_i})$.
The cladogram on $n$~particles embedded in continuous time can be summarized by 
the set of these non-negative sequences~$\{ (T^{(i)}_{j})_{j=1}^{k_i} \}_{i=1}^n$.
In Figure~\ref{fig:cladogram} both $A$ and $B$ are associated with the
sequence $(t_0, t_1)$, $C$ with $(t_0,t_1^\prime)$, and both $D$ and $E$
with $(t_0,t_1^\prime, t_2)$.

Here we translate Theorem~\ref{cty_pred_thm} to the setting of probability 
distributions on cladograms embedded in continuous time.  
The weak continuity condition is slightly different in this setting
as the distribution depends on the previous branch choices of 
the new particle. However, conditional on the branch up to a particular
point, the predictive distribution is weakly continuous if a small
perturbation of the future splitting times gives rise to a small
perturbation in the predictive distribution.
The theorem is a direct consequence of the equivalence of
equations~\eqref{cladcond1} and~\eqref{cladcond2} only 
when $\beta = 0$ (ie. the pilgrim process in this setting).
Theorem~\ref{branching_wkcty}
shows that in a large subclass of Markov branching models,
an analogous result to Theorem~\ref{cty_pred_thm} holds.

\begin{theorem}
\label{branching_wkcty}
A Markov branching process with splitting rule of the form
\[
q (i,j) \propto w(i) w(j)
\]
for some sequence of weights $w(j) \geq 0$, $j \geq 1$ has
weakly continuous predictive distributions if and only if it is either
the beta-splitting model with $\beta = 0$ or the comb
model ($\beta \to -2$).
\end{theorem}

The \emph{comb} model is the maximally unbalanced model where 
only one particle splits from the rest at each time.  This is 
analogous to the~$iid$ setting for Markov survival processes
where ties do not occur with probability one.
In the case of neutral evolutionary 
theory, the fragmentations are assumed to happen simultaneously.  Consider
the above example.  The set $\{ A,B \}$ is assumed to split from
the set $\{ C,D, E\}$ simultaneously.  However, it may be the case 
that there are two splits occurring in close temporal proximity.  
As all other Markov branching models have discontinuous predictive
distributions, the use of the beta-splitting model for $\beta = 0$
in applications where simultaneous splits are likely the result
of approximation seems natural.
Even considering the cladogram when it is not embedded 
in continuous time, equality of equations~\eqref{cladcond1}
and~\eqref{cladcond2} may seem an appropriate assumption for
a neutral evolutionary tree model.

In a data example, Aldous (1996) finds the ``fit to the $\beta = [0]$ 
model is strikingly better than to the usual models.'' Blum and Francois (2006)
support this claim with analysis of two additional datasets.
Aldous asks if this ``is just a fluke?''  The above findings suggest not.
Indeed they suggest the model analogous to the pilgrim process
may be most fitting.  
It should be noted that the branching topology is a non-smooth
function of the value of $\beta$ with
particularly different behavior for $\beta = 0$.  As many previous
conclusions in neutral evolutionary theory assume a uniform 
branch split model ($\beta = -1/2$) underlying the evolutionary tree, 
these claims may need to be revisited under this new paradigm.

\section{Extensions of the pilgrim process}
\label{extensions}

The pilgrim process can be extended to a three parameter
family indexed by $\rho>0$, $\beta>-1$, and $\nu >0$.  
Here we continue to assume $\nu = 1$.  
For $\beta = 0$, the pilgrim process
is recovered.  The generalized pilgrim process exhibits different
asymptotic behavior as a function of $\beta$.

\subsection{Generalized pilgrim process}

\begin{definition} \normalfont
Considering only the taxes, the \emph{generalized pilgrim process}
assumes the hotel tax for the~$i$th hotel is given by
\[
\log \left( \frac{ R(t_{i-1}) + \rho}{ R(t_{i}) + \rho + \beta} \right)
\]
where $t_0 = 0$.  
Note that $\beta = 0$ yields the hotel tax for the pilgrim process.
The toll fees, $\tau (s)$, now depend on the choice of $\beta$.
\begin{enumerate}
\item For $\beta > 0$,
\[
\tau ( s ) = B ( \rho + R(s) + 1, \beta ) - B( \rho + R(s), \beta)
\]
where $B(\cdot, \cdot)$ is the beta function.
\item For $\beta \in (-1, 0)$,
\[
\tau ( s ) = B (\rho + R(s), \beta +1)
\]
\end{enumerate}
While the taxes are no longer the ratio of upstream and downstream
tolls, the toll choice preserves exchangeability of the joint
distribution for the non-negative sequence $(T_1, \ldots, T_n)$
for all $n$.
\end{definition}

Behavior of the generalized pilgrim process depends
on the choice of $\beta$ (Dempsey and McCullagh, 2015).  For $\beta > 0$, the 
number of hotels~$N_n$ grows logarithmically.  Moreover, the fraction of
individuals occupying each hotel admits a stick-breaking representation.
If $P_i$ denotes the fraction of individuals in hotel~$i$ under spatial ordering,
then
\begin{equation} \label{stick_breaking}
(P_1, P_2, \ldots) \overset{D}{=} (W_1, \bar{W}_1 \cdot W_2, \ldots)
\end{equation}
where $W_i = \text{beta} (\beta, \rho)$ and $\bar{W}_i = 1 - W_i$.

On the other hand, for $\beta \in (-1,0)$, the number of hotels
grows polynomially, $N_n = O(n^{-\beta})$.  
If $\# B_{1,n}$ denotes the number of occupants of the first 
hotel after $n$ pilgrims then
\begin{align*}
\lim_{n \to \infty} \pr ( \# B_{1,n} = d ) &\to g_{\beta} (d) \\
&=\frac{-\beta}{\Gamma(1+\beta)} \frac{\Gamma(d+\beta)}{\Gamma(d+1)} 
\propto d^{-1 + \beta}
\end{align*}
for large $d$.  So the number of occupants has a power law distribution of degree $1-\beta$.

Lemma~\ref{occupancy} gives new asymptotic results for the number of hotels of a given
size for the generalized pilgrim process.

\begin{lemma} \label{occupancy}
For $n \geq 1$, let $N_{n,j}$ be the number of hotels with occupancy $j$ for a 
generalized pilgrim process after $n$ pilgrims.  
\begin{enumerate}
\item For $\beta > 0$, $N_{n,j} \to_{D} Y_j$ 
as $n \to \infty$ where $Y_j$ 
is a Poisson random variable with $E ( Y_j ) = \frac{1}{\psi(\rho+\beta)-\psi(\rho)} j^{-1}$.
For $\beta = 1$ this corresponds to condition $E (Y_j ) = \rho/j$.
\item For $\beta \in (-1,0)$, asymptotically $\# B_{i,n} \ci \# B_{j,n}$ for all fixed
$i, j >0$.  Moreover, $N_{n,j}/N_n \to g_\beta (j)$ almost surely.
\end{enumerate}
\end{lemma}

Proof of Lemma~\ref{occupancy} can be found in Appendix~\ref{app:occupancy}.
When $\beta > 0$, it follows from Lindeberg's theorem that
\begin{equation}
\label{asym1}
\frac{N_n - C_{\rho, \beta} \log (n)}{\sqrt{C_{\rho, \beta} \log (n) } } \to N (0,1)
\end{equation}
where $C_{\rho, \beta} = (\psi (\rho + \beta) - \psi (\rho) )^{-1}$. For $\beta = 1$ 
we have $C_{\rho, \beta} = \rho$.  For $\beta = 0$, we have a similar 
asymptotic result. Namely, there exists $C_\rho > 0$ such that
\begin{equation}
\label{asym2}
\frac{N_n - C_{\rho} \log^2 (n)}{\sqrt{C_{\rho} \log^2 (n) } } \to N (0,1).
\end{equation}
Equations~\eqref{asym1} and~\eqref{asym2} provide new asymptotic
normality results for the number of blocks of the generalized pilgrim process.

\subsection{Ordered partitions and splitting rules}
\label{split_rules}

A partial ranking of $[n]$ is an \emph{ordered list} $A = ( A_1, \ldots, A_k )$ consisting 
of disjoint non-empty subsets of $[n] = \{ 1, \ldots, n \}$ whose union is $[n]$.
The term partial ranking and ordered partition are used interchangeably.
The elements are unordered within blocks, but $A_1$ is the subset ranked first, $A_2$ 
the subset ranked second, and so on.  Let $\OP_n$ and $\OP$ denote the set of partial rankings
of $[n]$ and $\mathbb{N}$ respectively.
An ordered partition process $A$ is a probability distribution
on the infinite set $\OP$.

Restriction of the generalized pilgrim process to the first $n$ pilgrims
gives a joint distribution of the non-negative sequence $(T_1, \ldots, T_n)$. 
The non-negative sequence is in one-to-one correspondence 
with a pair of a partial ranking $A[n] \in \OP_n$ and interarrival 
times $S = (S_1, \ldots, S_k) \in \Real^k$
\[
(T_1, \ldots, T_n) \leftrightarrow A \times S
\]
where $k = \# A$ is the number of blocks.  

Here we are only interested in the partial ranking determined by 
the temporal ordering of the hotels and the number of occupants
of the generalized pilgrim process.  It is completely determined by a
splitting rule.  The chance that $r$ pilgrims can pay the total tax
of a hotel and $d$ cannot is given by
\begin{equation}
\label{splitting_rules}
q(r,d) = \frac{1}{\zeta_{r+d}} \frac{ \Gamma ( r+\rho ) \Gamma ( d + \beta)}{\Gamma (r+d + \rho + \beta)}
\end{equation}
where $\zeta_{r+d}$ is the normalizing constant.  For example,
if $\beta = 1$ then the splitting rule becomes
\begin{equation}
\label{split_beta1}
\frac{ \Gamma (r + \rho) \Gamma( d + 1 ) }{ \Gamma (r+d+\rho +1)} \frac{ r+d + \rho}{ (r + d) B(1, \rho)}
= \rho \frac{ \Gamma (r + \rho) \Gamma( d  ) }{ \Gamma (r+d+\rho )} \frac{d}{r + d}
\end{equation}
where $B(\cdot, \cdot)$ is the beta function.   
The splitting rule allow us to easily infer the induced probability distribution 
on partial rankings discussed in Section~\ref{clustering}.

Using the associated splitting rules,
the distribution of $A$ is given by
\begin{equation} \label{splitrules}
\pr ( A_1 , \ldots , A_k ) = \prod_{i=1}^k q( r_i, d_i )
\end{equation}
where $r_i = \sum_{j>i} \# A_j$ and $d_i = \# A_i$.
Equation~\eqref{splitrules} is important in establishing
the connection between the pilgrim process and Ewens
sampling formula.

\subsection{Pilgrim's voyage and recurrent events}

We now consider the process where each pilgrim must reach a final destination, $T^\star < \infty$.
Assume that pilgrim~$i$ has access to a sequence of funds, $(X_{i,1}, X_{i,2}, \ldots)$.

Initially, there are no hotels, the toll rate is uniform $1/\rho$ dollars per mile,
so the first traveller setting out from the origin
proceeds to the point $T_{1,1} = \rho X_{1,1}$ where he establishes the first hotel.
He then embarks on a second voyage with funds $X_{1,2}$ and establishes
a second hotel at $T_{1,2} = T_1 + \rho X_{1,2}$.  He continues in this manner
until reaching the final destination, $T^\star$.

At this stage, the toll rate at \emph{all} points is reduced to $1/(1+\rho)$,
and the hotel tax for all hotels is equal to the log ratio $\log((1+\rho)/\rho)$.
After $n$ pilgrims have set out from the origin
and reached the final destination, there are a set of 
established hotels which are the unique times of the set~$\{ T_{i,j} \}$.
Suppose hotel~$k$ has been occupied by $d_k \ge 1$ pilgrims,
then for the $n+1$st pilgrim the toll rate from $(0,T^\star)$ is constant equal to $1/(\rho + n)$,
while the hotel~$k$ tax is the log ratio
\[
\log \left( \frac{n + \rho}{n - d_k + \rho} \right).
\] 
This generates an infinitely exchangeable doubly-indexed 
non-negative sequence $\{ T_{i,j} \}$.  

The joint density function for 
the restriction of the sequence to $i \leq n$
is
\[
\rho^{\uparrow n} f_n (\T) = \exp \left( - T^\star \zeta ( n ) \right) \prod_{k=1} \Gamma( d_k )
\]
The pilgrim's voyage is useful for describing
sequences of recurrent event times.
The pilgrim's voyage can be generalized in a similar way as the 
pilgrim's monopoly.  Moreover, $\rho$ can depend on time as well 
as the number of previous event times.  This implies such a description
can easily be extended to incorporate time-inhomogeneity 
as well as historical dependencies.
Section~\ref{voyage_to_ibp} highlights the connection between
the pilgrim's voyage and the Indian buffet process.


\section{Exchangeable clustering and feature allocation models}
\label{clustering}


The goal of cluster analysis is to identify a set of mutually exclusive
and exhaustive subsets $(B_i)_{i=1}^k$ of a finite
set~$[n]$.  The idea is points in the set 
$B_i$ will represent a
homogeneous sub-population of the sample.  
The problem is an example of unsupervised learning
where there is no external information to aid in 
the decision making process.

As labelling of elements and the blocks are completely arbitrary, 
it is natural to disregard the labels and focus solely on the
underlying partition of the index set.  To this end, 
exchangeable cluster processes have been 
developed.
Below connects the current study of the pilgrim process
to the study of exchangeable partition processes.  
A connection to the Ewens sampling formula
is derived.

\subsection{Partition processes and partition structures}

A partition of $[n]$ is a set $B = \{ B_1, \ldots, B_k \}$ consisting of disjoint
non-empty subsets of $[n] = \{ 1, \ldots, n \}$ whose union is $[n]$.
The elements of $[n]$ are unordered within or among blocks.
Let $\P_n$ denote the set of partitions of $[n]$, and $\P$ the set of
partitions on the infinitely countable set $\mathbb{N}$.  
Let $\#B = k$ denote the number of blocks, and $\# b$ the size
of block $b \in B$.  The alternative notation $B = B_1 | B_2 | \cdots | B_k$ is
commonly used.



The Chinese restaurant process is a sequential description of a particular
infinitely exchangeable partition process.

\begin{definition}
Let $B[n]$ denote the restriction of a partition process $B \sim p$ to the $[n]$. 
Then the \emph{Chinese restaurant process} is defined sequentially.  
Imagine a restaurant with countably infinite tables indexed by $\mathbb{N}$.  
The tables are chosen accordingly:
\begin{enumerate}
\item The first customer sits at the first table.
\item The placement of the $n+1$st customer is determined by the conditional
distribution $p_{n+1} ( B[n+1] | B[n] )$.  The customer chooses a new table with probability
$\theta / (\theta + n)$ or one of the occupied tables with probability proportional to
the number of occupants.
\end{enumerate}
\end{definition}

The probability of a particular partition, $B[n]$, is given by
\[
p_n ( B[n] ) = \frac{  \theta^{\# B} \prod_{b \in B} \Gamma ( \# b) }{ \theta^{\uparrow n}} 
\]
where $\theta^{\uparrow k} = \theta \cdot ( \theta + 1 ) \cdots ( \theta+ k - 1)$ is 
the ascending factorial. 
This is the celebrated Ewens sampling formula, which arises in a variety of settings (Crane, 2015).

\subsection{From pilgrim process to Ewens sampling formula}

Consider generating a random partition $B \sim p$ via the generalized pilgrim process 
by extracting the ordered partition $A$ from the non-negative sequence $T$
and then simply ignoring the ordering of blocks. The map is not reversible, however the resulting
random partition is infinitely exchangeable.
\begin{equation} \label{seq_to_part}
T \to A \to B
\end{equation}
The probability of the random partition $B[n]$ is then
\begin{equation} \label{ordpar_to_part}
p_n ( B_1 | \ldots | B_k ) = \sum_{\sigma \in S_k} p_n (A_{\sigma(1)}, \ldots, A_{\sigma(k)} )
\end{equation}
where $S_k$ is the group of permutations of $k$ items.  Note the probabilities on the 
right-hand side are on partial rankings in $\OP_n$ while the left-hand side is a probability
on partitions in $\P_n$.

Theorem~\ref{crp_correspondence} provides an interesting connection between 
the Ewens sampling formula and the generalized pilgrim process.

\begin{theorem}
\label{crp_correspondence}
The random partition $B[n]$ obtained from the generalized pilgrim process for $\beta = 1$ and $\rho >0$
is equivalent in distribution to a random partition $B \in \P_n$ generated by the Ewens sampling formula
for $\theta = \rho >0$.
\end{theorem}

\begin{proof}
The splitting rule for $\beta = 1$ is given by equation~\eqref{split_beta1}. 
By equation~\eqref{splitrules}, this implies the probability of a particular partial
ranking is 
\begin{align*}
p_n ( A_1, \ldots, A_k ) &= \rho^k \prod_{j=1}^k \frac{ \Gamma (d_j +\beta) \Gamma(r_j + \rho) }{ \Gamma ( r_j + d_j + 1+ \rho) } \frac{\rho + r_j + d_j }{r_j + d_j } \\
&= \frac{\rho^k \Gamma ( \rho)}{\Gamma( \rho + n)} \prod_{j=1}^k \Gamma ( d_j ) \cdot \prod_{j=1}^k \frac{d_j}{r_j + d_j} \\
&= p_n ( B_1 | \ldots | B_k ) \cdot \prod_{j=1}^k \frac{d_j}{r_j + d_j}
\end{align*}
where cancellation is a consequence of $r_{j} = r_{j+1} + d_{j+1}$.  Without loss of generality, we 
assume this is the canonical labeling ($\sigma  =$ identity) of the ordered partition.  Applying 
equation~\eqref{ordpar_to_part}, the proof is complete if we can show
\[
\sum_{\sigma \in S_k} \prod_{j=1}^k \frac{ d_{\sigma (j)} }{ d_{\sigma (j) } + r_{\sigma (j)} } = 1
\]
where $d_{\sigma(j)} = \# A_{\sigma(j)}$ and $r_{\sigma (j)} = \sum_{\sigma (i) > \sigma (j) } \# A_{\sigma(i)}$.

Consider sampling without replacement from $k$ particles with weights proportional to $(d_1, \ldots, d_k)$.  
Then the probability of a given order of sampling of the particles, $\sigma \in S_k$, is given by
\[
\pr ( \sigma ) = \prod_{j=1}^k \frac{ d_{\sigma (j)} }{ d_{\sigma (j) } + r_{\sigma (j)} } 
\]
Therefore by law of total probability we have the sum over all possible orderings is
equal to one as desired.
\end{proof}

While the mapping given by equation~\eqref{seq_to_part} is noninvertible, the proof
to equation~\eqref{crp_correspondence} provides a means for taking a random
partition from the Ewens sampling formula and generating a random ordered
partition equivalent in distribution to that generated by the generalized pilgrim
process by sampling without replacement from the blocks~$b \in B$ with probabilities 
proportional to the size of the blocks $\# b$.

The stick-breaking representation given by equation~\eqref{stick_breaking}
shows that asymptotic frequencies of the first block of the random ordered
partition and the first block of the Chinese restaurant process with blocks
in size-biased order are equal in distribution.  However, sequential construction
via the generalized pilgrim process shows that for each $n$, there is a non-zero
chance that the first block becomes a singleton.  Therefore, while asymptotically
equivalent the finite-dimensional distributions are not.  We now state a 
straightforward corollary to Theorem~\ref{crp_correspondence}.

\begin{corollary}
The one-parameter Chinese restaurant process is a subfamily of the
partition process induced by the generalized pilgrim process. 
\end{corollary}

For $\beta >0$, the two-parameter extension of the one-parameter Chinese
restaurant process allows us to control the relative frequencies.  That is,
the expected asymptotic frequencies for the first and second blocks 
are $\beta/ (\rho + \beta)$ and $\rho \beta/ (\rho + \beta)^2$ respectively.  So the 
parameter $\beta$ controls the ratio of the expected asymptotic frequencies, $\rho/ (\rho + \beta)$.

For $-1 < \beta < 0$ and $\rho > 0$, the construction is similar to the
two-parameter Chinese restaurant process.  For $0 < \alpha < 1$ and $\theta >0$, 
the two-parameter seating model has customers choosing a new table with
probability proportional to~$\theta + \alpha \, \# B$ or one of the occupied
tables~$b \in B$ with probability proportional to $\# b - \alpha$.
The number of blocks grows at the same
asymptotic rate, $n^{\alpha}$.  
Lemma~\ref{twoparam_lemma} is a consequence of 
Lemma~\ref{occupancy} and Lemma 3.11 of Pitman (2005, pg. 71)
which provides an asymptotic distribution for the 
proportion of blocks of size~$k$ for the two-parameter seating model.

\begin{lemma}
\label{twoparam_lemma}
For $\beta = - \alpha \in (-1, 0)$ and $\theta = \rho > 0$ the 
induced partition from the generalized pilgrim process 
converges in distribution to the distribution of the random
partition generated by the two-parameter Chinese restaurant process.
\end{lemma}

An interesting class of priors is now generated by looking at the pilgrim process
($\beta = 0$).  
The number of blocks grows at $\log^2 (n)$ which does not correspond 
to any of the seating models discussed. 
Therefore, the pilgrim process provides a method for generating an
exchangeable partition processes with intermediary behavior between
the one-parameter Chinese restaurant process  ($\alpha = 0$) and the 
two-parameter process ($\alpha \in (0,1)$).
We can think of the pilgrim process as an embedding in continuous time
of a partition process.  Clusters that are similar will be in close temporal
proximity.
Weak continuity of conditional distributions suggests that
the difference between two clusters being arbitrarily close to one another
versus one larger cluster on the predictive distribution is negligible at continuity points.
In situations where such behavior is desired, the pilgrim process seems
a viable candidate for cluster analysis. 

\subsection{Feature allocation processes}

A \emph{feature allocation} is a generalization of a partition that relaxes
the constraint of mutual exclusivity of the blocks. In particular, a 
feature allocation of $[n]$ is a multiset $F = \{ F_1, \ldots, F_k \}$
consisting of non-empty subsets of $[n]$.  Like a partition the 
elements of $[n]$ are unordered within blocks, and blocks are 
unordered as well.  Let $\F_n$ denote the set of feature allocations
of $[n]$, and $\F$ the set of feature allocations of $\mathbb{N}$.
The \emph{Indian buffet process} is a sequential description of a particular 
feature allocation process. 

\begin{definition} \normalfont
Let $B[n]$ denote the restriction of a feature allocation process $B \sim p$ to $[n]$. 
Then the \emph{Indian buffet process} is defined sequentially.  
Imagine a restaurant with countably infinite dishes indexed by $\mathbb{N}$.  
Each customer chooses a finite number of dishes according to the following
:
\begin{enumerate}
\item The first customer chooses $K_1^{+} \sim \text{Poisson} (\gamma)$ dishes
\item The number of dishes for the $n+1$st customer is determined by the conditional
distribution $p_{n+1} ( B[n+1] | B[n] )$.  For a previously sampled dish~$k$, 
the customer samples the dish with probability
\[
\frac{N_{n,k} - \alpha}{\theta + n}
\]
where $N_{n,k}$ denotes the number of previous customers to have sampled the dish.
The customer then chooses 
\[
K_{n+1}^{+} \sim \text{Poisson} \left( \gamma \frac{ \Gamma ( \theta + 1)}{\Gamma( \theta + n+1)} \frac{ \Gamma (\theta + \alpha + n)}{\Gamma (\theta + \alpha)} \right)
\]
new dishes. If $K_{n+1}^{+} > 0$ the new dishes are labelled by 
order-of-appearance $K_{n} + 1, \ldots, K_{n+1}$.
Here $K_n$ represents the number of dishes sampled after $n$
customers, $K_n = K_{n-1} + K_n^{+}$ (with $K_0 = 0$).
\end{enumerate}
\end{definition}

\subsection{Pilgrim's voyage to random feature allocations}
\label{voyage_to_ibp}

Consider the doubly-indexed non-negative sequence $\{ T_{i,j} \}$ associated 
with the pilgrim's voyage.  Then similar to how the pilgrim process induces
an ordered partition, $A \in \OP$, the pilgim's voyage induces an ordered feature 
allocation, $A^\prime \in \OF$.  The mapping defined by equation~\eqref{seq_to_part}
extends naturally to this setting for feature allocation processes.

The feature allocation, $B^\prime[n]$, induced by the pilgrim's voyage has a 
sequential description.  For $n=1$, the number of new hotels, $K_1^{+}$,
is a Poisson process with rate parameter $1/\rho$.   Given a final destination
$T^\star$, $K_1^{+} \sim \text{Poisson} ( T^\star / \rho )$.

For the $n+1$st pilgrim, the chance he stays at a hotel is $d / (n+\rho)$
where $d$ denotes the number of previous occupants.  The number of 
new hotels $K_{n+1}^{+}$ is given by a Poisson process with rate $1/(n + \rho)$. 
Therefore $K_{n+1}^{+} \sim \text{Poisson} ( T^{\star} / (n+ \rho) )$.

Here we have assumed $\nu = 1$.  Otherwise, the sequential
description changes as the rate of the Poisson process is multiplied
by $\nu$.
Without loss of generality for the induced feature allocation process, 
we set $T^{\star} = 1$ and $\nu > 0$.
Theorem~\ref{ibp_to_pv} establishes the connection between the
Indian buffet process and the pilgrim's voyage.
 
\begin{theorem} 
\label{ibp_to_pv}
The random feature allocation $B^\prime[n] \in \F_n$ induced by the pilgrim's voyage is equal
in distribution to the random feature allocation $B \in \OF_n$ generated
by the Indian buffet process for $\theta = \rho$, $\gamma = \nu / \rho$, and
$\alpha = 0$.
\end{theorem}

The proof is a direct consequence of the equivalence between the sequential
descriptions of the random feature allocation processes.
If we ignore the ordering of the pilgrim's voyage for $\beta >0$ and $\rho = 1$ we have the
process that corresponds to the Indian buffet process.
The pilgrim's voyage can be generalized in the same way as the 
pilgrim process.  The generalization is a three-parameter family
with parameters $(\nu, \rho, \beta)$ such that $\nu, \rho >0$ and
$\beta > -1$.  By similar arguments via sequential descriptions 
one can prove the following lemma.

\begin{lemma}
Given the Indian buffet process with parameters $\gamma > 0$, $\alpha \in [0,1)$, 
and $\theta >0$, there exists $\nu > 0$, $\beta \in (-1,0]$, and $\rho > 0$
such that the random feature allocation induced by the generalized pilgrim's voyage
is equivalent in distribution to that obtained by the Indian buffet process.
\end{lemma}

In particular, the three-parameter Indian buffet process is a proper subset
of the three-parameter family of feature allocation processes generated
by the generalized pilgrim's voyage.

\section{Summary}

In this paper, we presented the pilgrim process, a probabilistic process 
in which initial funds $X_1, X_2,\ldots$ for each pilgrim gives rise to a 
non-negative sequence $T_1, T_2,\ldots$ that is infinitely exchangeable.  
Simple expressions for a sequential description of the process, the joint density, 
and the multivariate survivor function are provided as well as a connection 
to the Kaplan-Meier estimator.  Simulation allows investigation of the 
process's properties; for example, we find that the number of hotels grows 
in tandem with the total expenditure on tolls, both growing at a rate 
proportional to $\log^2 (n)$.  We connect the pilgrim process with the class 
of survival processes investigated by Hjort (1990), seeing that the pilgrim process 
is closely approximated by the gamma process originally considered by 
Kalbfelisch (1978) and Clayton (1991).  Also, generalized pilgrim process 
is connected to the study of Markov branching models in neutral 
evolutionary theory.  We end with exploration of extensions to the 
pilgrim process and its connection to exchangeable partition 
and feature allocation processes.

\appendix 

\newpage

\section{Technical Arguments}

\subsection{Proof of Theorem~\ref{cty_pred_thm}}
\label{app:cty_pred}

Here is presented the conditions for the conditional survival function, $\text{pr} \left( T_{n+1} > t | R[n] \right)$, to be a continuous function of the observed risk trajectory.  
To each consistent splitting rule there is an associated characteristic index $(\zeta_{i} )_{i=1}^\infty$ defined 
by
\[
\zeta_{n+1} = \frac{\zeta_n}{1- q(n,1)} = \zeta_1 \prod_{j=1}^n (1-q(j,1))^{-1}
\]
The splitting rules can be described by $k$th order differences
\[
(\Delta^k \zeta)_n = \sum_{j=1}^k (-1)^{k-j} {k \choose j} \zeta_{n+j}
\]
Then
\[
q(r,d) = \frac{ (-1)^{d+1} (\Delta^d \zeta)_r }{ \zeta_{r+d} }
\]
For the conditional survival function to be right continuous, the atomic component of the conditional survival distribution must satisfy
\[
\frac{ (\Delta^{d-1} \zeta ) (r+2) } {(\Delta^{d-1} \zeta ) (r+1)} \cdot \frac{ (\Delta \zeta ) (r+1) } {(\Delta \zeta ) (r)} = \frac{ (\Delta^{d} \zeta ) (r+1) } {(\Delta^{d} \zeta ) (r)}
\]
for $r \geq 0$ and $d >1$.  On the other hand, for the function to be left continuous, the atomic component must satisfy
\[
 \frac{ (\Delta \zeta ) (r+d) } {(\Delta \zeta ) (r+d-1)} \cdot \frac{ (\Delta^{d-1} \zeta ) (r+1) } {(\Delta^{d-1} \zeta ) (r)} = \frac{ (\Delta^{d} \zeta ) (r+1) } {(\Delta^{d} \zeta ) (r)}
\]
Recursive substitution shows the two conditions to be equal and continuity holds if
\begin{equation} \label{cty_cond}
 \frac{ (\Delta \zeta ) (r+d) } {(\Delta \zeta ) (r)} = \frac{ (\Delta^{d} \zeta ) (r+1) } {(\Delta^{d} \zeta ) (r)}
\end{equation}
It is now shown that such a condition is uniquely satisfied by the harmonic process.

\begin{proposition}
The harmonic process is the only Markov survival process with continuous conditional survival function.  That is, 
\[
\lim_{T_i \to t^\prime} \text{pr} \left( T_{n+1} > t | R[n] \right) = \text{pr} \left( T_{n+1} > t | R^\prime [n] \right)
\]
where $R^\prime$ is the risk set trajectory when $T_i = t^\prime$ for some $i \in [n]$ {\bf if and only if} $T_{n+1}$ is from a harmonic survival process.
\end{proposition}

\begin{proof}
Recall the definition of the $k$th order forward differences
\[
\lambda(r,d) = (\Delta^d \zeta) (r) = \sum_{j=0}^{d} (-1)^{d-j} {d \choose j} \zeta_{r+j}
\]
which implies that this forward difference is a linear function of the set $\{ \zeta_r, \ldots, \zeta_{r+d} \}$. 

Start by considering the standardized characteristic index ($\zeta_1 = 1$).  First, let $\zeta_2 > 0$ be fixed.  

Let $n = r+d > 2$ and $d \geq 1$. Then equation~\eqref{cty_cond} becomes
\begin{align*}
 \frac{ (\Delta \zeta ) (n) } {(\Delta \zeta ) (r)} &= \frac{ (\Delta^{d} \zeta ) (r+1) } {(\Delta^{d} \zeta ) (r)} \\
  \Rightarrow (\Delta \zeta ) (n) \cdot (\Delta^{d} \zeta ) (n-d) &=  (\Delta^{d} \zeta ) (n-d+1) \cdot (\Delta \zeta ) (n-d)
\end{align*}
By definition, $(\Delta \zeta ) (n)$ is a linear function of $\{ \zeta_{n}, \zeta_{n+1} \}$ while $(\Delta^d \zeta ) (r)$ is a function of $\{ \zeta_r, \ldots, \zeta_{n} \}$.  Therefore, the left hand side is a linear function of $\zeta_{n+1}$ given $\{ \zeta_i \}_{i \leq n}$.  On the right hand side, $(\Delta^d \zeta) (n-d+1)$, is a function of $\{ \zeta_{n-d+1}, \ldots, \zeta_{n+1} \}$ while $(\Delta \zeta ) (n-d)$ is a function of $\{ \zeta_{n-d} , \zeta_{n-d+1} \}$.  Since $d \geq 1$, both sides are linear in $\zeta_{n+1}$.  Therefore solving for $\zeta_{n+1}$ shows the characteristic index is a deterministic function of the previous characteristic index values $\{ \zeta_{1}, \ldots, \zeta_n \}$.

The above argument holds if the coefficient of $\zeta_{n+1}$ is nonzero.  The coefficient is equivalently zero if $\lambda(n-d,d) = \lambda(n-d,1)$.  Now, suppose equality holds for all $d \in [n]$.  By definition, the splitting rule satisfies
\begin{align*}
1 &= \sum_{d=1}^n {n \choose d} q(n-d,d) \\
\zeta_n  &= \sum_{d=1}^n {n \choose d} \lambda (n-d,d)  \text{ \hspace{0.2cm} by definition }\\
\zeta_n  &= \sum_{d=1}^n {n \choose d} \lambda (n-d,1) \text{ \hspace{0.2cm} by assumption }\\
\zeta_n  &= \sum_{d=1}^n {n \choose d} \left( \zeta_{n-d+1} - \zeta_{n-d} \right) \text{ \hspace{0.2cm} by definition } \\
\zeta_n &= \frac{1}{n-1} \left[ n\cdot \zeta_{n-1} - \sum_{d=2}^n {n \choose d} \left( \zeta_{n-d+1} - \zeta_{n-d} \right) \right]
\end{align*}
which implies that $\zeta_n$ is again a function of the previous characteristic indices.

The above argument shows $\zeta_k$ is a deterministic function of $\{ \zeta_1, \zeta_2 \}$ for $k \geq 3$.  However, $\zeta_2$ is constrained by choice of $\zeta_1$.  In particular, the holding times satisfy $\zeta_{n+1} = \zeta_n / (1- q(n,1) )$.  Moreover, the splitting rule must satisfy $0 \leq q(n,1) \leq 1/(n+1)$.  Therefore
\[
\zeta_n \leq \zeta_{n+1} \leq \left( 1+\frac{1}{n} \right) \zeta_{n}
\]
This translates to $\zeta_2 = \zeta_1 (1 + c)$ for some $c \in [0,1]$.  It rests to show a correspondence between $c$ and the parameter $\rho$ controlling the harmonic process.

For the harmonic process,
\[
\nu \Gamma(1) / \rho = \zeta_1
\]
so that $\zeta_1 \rho = \nu$.  Then due to the normalization of the splitting rule
\begin{align*}
\zeta_2 &= {2 \choose 1} \frac{\nu \Gamma (1)}{(1+\rho)} + {2 \choose 2} \frac{\nu \Gamma (2)}{\rho \cdot (1+\rho)} \\
&= \zeta_1 \left[ 1+ \frac{\rho}{1+\rho} \right]
\end{align*}
This establishes the correspondence between $c$ and $\rho$ and therefore between all $\{ \zeta_1, \zeta_2 \}$ that define continuous survival functions and the set of all harmonic processes with parameters $\rho$ and $\nu$.
\end{proof}

\begin{proposition}
If $\lambda(m-d,d) = 0$ for fixed $m \geq 1$ and for all $2 \leq d \leq m$, then $\zeta_n \propto n$ for all $n$.  Thus the only process with trivial forward differences is when $T_i$ are iid exponential.
\end{proposition}

\begin{proof}
Consider the standardized sequence, $\zeta_1 = 1$, and start by showing that $\lambda(m-d,d) = 0$ implies that $\zeta_k = k$ for $k \leq m$.  Indeed, the condition implies that $\zeta_m = m$ and $q(m-1,1) = 1/m$.  Consistency implies
\begin{align*}
(1-q(m-1,1)) \cdot q(m-1-d,d) &= q(m-1-d,d+1) + q(m-1-d+1,d) \\
&= q(m-(d+1), (d+1) ) + q(m - d, d) \\
&= q(m-d,d)  \\
&= \bigg \{ \begin{array}{c c} 0 & \text{if } d \geq 2 \\ 1/m & \text{if } d = 1 \end{array}
\end{align*}
This implies $q(m-1-d,d) = 0$ for $d \geq 2$ which implies $q((m-1) -1,1) = 1/(m-1)$ which is equivalent to $\zeta_{m-1} = m-1$.  Recursively applying this argument yields $\zeta_k = k$ for all $k \leq m$.

It rests to show that $\zeta_k = k$ for $k < n$ implies $\zeta_{n} = n$.  The normalization condition of the splitting rule implies
\[
\sum_{d=1}^n {n \choose d} (\Delta^d \zeta) (n-d) = \zeta_n
\]
Now study the $d$th forward differences:
\begin{align*}
(\Delta^d \zeta) (n-d) &= \sum_{k=0}^d (-1)^{d-k} {d \choose k} \zeta_{n-d+k} \\
&= \zeta_n + \sum_{k=0}^{d-1} (-1)^{d-k} {d \choose k} (n-d+k) \\
&= \zeta_n -n + {\bf 1} [d = 1] \hspace{0.2cm} \text{ by Lemma 0.3}
\end{align*}
Plugging into the normalization condition yields
\begin{align*}
\zeta_n &= \sum_{d=1}^n {n \choose d} \left[ \zeta_n -n + {\bf 1} [d = 1] \right]  \\
&= 2^n \zeta_n - n \cdot (2^n -1) \\
\Rightarrow \zeta_n &= n
\end{align*}
which completes the proof.
\end{proof}

\begin{lemma} \label{lemma1}
\[
\sum_{k=0}^{d-1} (-1)^{d-k} {d \choose k} (n-d+k) = -n + {\bf 1} [d = 1]
\]
\end{lemma}
\begin{proof}
First note $\sum_{k=0}^{d} (-1)^{d-k} {d \choose k} = 0$ and therefore 
\[
\sum_{k=0}^{d-1} (-1)^{d-k} {d \choose k} (n-d) = -(n-d)
\]
It was previously shown that $\sum_{k=0}^{d} (-1)^{d-k} {d \choose k} k = {\bf 1} [d = 1]$ which implies
\[
\sum_{k=0}^{d-1} (-1)^{d-k} {d \choose k} k  = \bigg \{ \begin{array}{c c} 0 & \text{if } d= 1 \\ -d & \text{if } d \neq 1 \end{array}
\]
Adding these together proves the lemma.
\end{proof}

\subsection{Proof of Theorem~\ref{branching_wkcty}}

Theorem 2 from McCullagh, Pitman, and Winkel (2008) 
proves that the binary fragmentation trees of Gibbs type
are identified with the Aldous' beta-splitting model.

Among the beta-splitting models, the key continuity condition
is exactly the same as stated in section~\ref{app:cty_pred}
\[
\frac{ (\Delta \lambda ) ( i + j)}{ (\Delta \lambda) (j)} = 
\frac{ (\Delta^i \lambda ) ( j + 1)}{ (\Delta^i \lambda) (j)} 
\]
which is only satisfied by the beta-splitting model with
$\beta = 0$ and as $\beta \to \infty$.


\subsection{Proof of Lemma~\ref{pilg_pois_lemma}}
\label{app:pilg_poisson}

The splitting rule for the harmonic process is 
\[
q(r,d) = \frac{1}{\psi(r+d+\rho) - \psi(\rho)} \cdot \left( 1 - \frac{d}{n} \right)^{\rho - 1} \frac{1}{d}
\]
When $n = r+d$ is large, $\psi(n+\rho) - \psi(\rho) \approx \log(n)$.  Moreover, the number of 
blocks grows at a rate of $\log^2 (n)$.  Therefore, by a Poissonization argument the 
number of blocks of size $d$ is approximately Poisson with rate parameter
\[
\lambda \approx c \cdot \frac{\log^2(n)}{\log (n)} \frac{1}{d} = c \cdot \log (n) \frac{1}{d} 
\]
where $c = \frac{1}{2 \Psi_1 (\rho) }$ is a constant dependent on $\rho$, and $\Psi_1$ is the 
trigamma function. 

\subsection{Proof of Lemma~\ref{occupancy}}
\label{app:occupancy}

The splitting rule for $\beta > 0$ is given by
\begin{align*}
q(r,d) &= \frac{1}{B(\rho, \beta) ( 1 - \rho^{\uparrow n} / (\rho + \beta)^{\uparrow n} ) } \cdot \left( 1 - \frac{d}{n} \right)^{\rho - 1} \left( \frac{d}{n} \right)^{\beta - 1} \frac{1}{n} \\
&\approx d^{-1} \frac{1}{B(\rho, \beta) } \cdot \left( 1 - \frac{d}{n} \right)^{\rho - 1} \left( \frac{d}{n} \right)^{\beta} \\
\end{align*}
Let $n_i = r_i + d_i$ and $d > 0$ a fixed constant.  Then define
\[
a_i = \frac{1}{B(\rho, \beta) } \cdot \left( 1 - \frac{d}{n_i} \right)^{\rho - 1} \left( \frac{d}{n_i} \right)^{\beta}
\]
Then
\[
\frac{1}{d} \sum a_i
\]
denotes the expected number of blocks of size $d$ 
where the random sequence $(n_1, n_2, \ldots)$ satisfies $\sum n_i = n$.
As $n \to \infty$, we want to show $\sum_i a_i < \infty$.  
By the ratio test
\begin{align*}
\left( a_{i} \right)^{1/i} &\to\left( \frac{1}{ n_{i} } \right)^{\beta/i} \\
&\approx n_{i}^{-\beta / (c \cdot \log(i)) } \\
&\to e^{-\beta / c } < 1 \text{ for } \beta, c > 0
\end{align*}
where we have used $n_i \sim c \log(i)$ for $c > 0$.
Therefore, the sequence converges almost surely and we have
the expected number of blocks of size $d$ is proportional to $d^{-1}$.

Since the expected number of blocks is $1/ (\psi(\beta + \rho) - \psi (\rho)) \log (n)$
we see that the expected number of blocks of size $d$ is 
\[
\frac{1}{\psi(\beta + \rho) - \psi (\rho)} \, d^{-1}
\]
In particular, for $\beta = 1$ then we have the expected 
number of blocks is $\rho / d$ which is equivalent to the one
parameter Chinese restaurant process.

\end{document}